\documentclass[10 pt, reqno]{amsart}
\usepackage{amsmath,amssymb,mathrsfs,graphicx,bbm,amscd}
\usepackage{amsfonts,amsthm}
\usepackage{xargs}
\usepackage{pgf,tikz}
\usepackage{color}
\usepackage{enumitem}
\usepackage{mathpazo}
\usepackage[T1]{fontenc}
\usepackage{fullpage}

\newtheorem{thm}{Theorem}[section]

\newtheorem{lem}[thm]{Lemma}
\newtheorem{defn}[thm]{Definition}
\newtheorem{prop}[thm]{Proposition}

\newtheorem{cor}[thm]{Corollary}

\newtheorem{rmk}[]{Remark}
\newcommand{\be}{\begin{eqnarray}}
\newcommand{\ee}{\end{eqnarray}}
\newcommand{\ben}{\begin{eqnarray*}}
\newcommand{\een}{\end{eqnarray*}}
\newcommand{\beal}{\begin{aligned}}
\newcommand{\enal}{\end{aligned}}
\newcommand{\beq}{\begin{equation}}
\newcommand{\eeq}{\end{equation}}

\newcommand{\eps}{\varepsilon}

\newcommand{\T}{\mathbb{T}}
\newcommand{\R}{\mathbb{R}}

\newcommand{\cC}{\mathcal{C}}

\usepackage[unicode=true]{hyperref}
\hypersetup{
     colorlinks,
     linkcolor={black!10!blue},
     linkbordercolor = {black!100!blue},
     citecolor={red}
}

\usepackage{enumitem}
\makeatletter
\def\namedlabel#1#2{\begingroup
    #2%
    \def\@currentlabel{#2}%
    \phantomsection\label{#1}\endgroup
}
\makeatother

\title{\textsc{On the regularity of stochastic effective Hamiltonian}}
 \author{Son N.T. Tu$^*$ and Jianlu Zhang$^\dagger$}
 \address{$^\dagger$Hua Loo-Keng Key Laboratory of Mathematics \& Mathematics Institute\\Academy of Mathematics and systems science\\Chinese Academy of Sciences, Beijing 100190, China}
 \email{jellychung1987@gmail.com}
 \address{$^*$Department of Mathematics, Michigan State University\\
 East Lansing, Michigan 48824, USA}
     \email{tuson@msu.edu}

\thanks{$\dagger$ {\it Statements and Declarations: }The authors declare no competing interests.\\ The work of Son Tu is supported in part by the NSF DMS grant 2204722. The work of Jianlu Zhang is supported by the National Key R\&D Program of China
(No. 2022YFA1007500) and the National Natural Science Foundation of China (No. 12231010). }
\date{\today}
\begin{document}

\begin{abstract} 
In this paper, we study the regularity of the ergodic constants for the viscous Hamilton--Jacobi equations. We also estimate the convergence rate of the ergodic constant in the vanishing viscosity process. 
\end{abstract}

\maketitle

\section{Introduction}
Let $\mathbb{T}^n=\mathbb{R}^n/ \mathbb{Z}^n$ be the flat torus. Under some assumptions on the Hamiltonian $H(x,\xi):\mathbb{T}^n \times \mathbb{R}^n\to \mathbb{R}$, to each $p\in\mathbb{R}^n$ there exists a \emph{unique} constant $\overline{H}(p)$ such that the following cell (or ergodic) problem
\begin{equation}\label{eq:Intro:cell-ergodic}
    H(x,p+Du(x)) = \overline{H}(p), \qquad 
x\in\mathbb{T}^n
\end{equation}
can be solved with a viscosity solution $u\in C(\mathbb{T}^n,\R)$. As the elliptic regularization of \eqref{eq:Intro:cell-ergodic}, there exists a unique constant $\overline{H}^\varepsilon(p)$ for every $\varepsilon>0$ such that
\begin{equation}\label{eq:Intro:cell-ergodic-2nd-order}
    H(x, p + Du^\varepsilon(x)) - \varepsilon \Delta u^\varepsilon(x) = \overline{H}^\varepsilon( p), \qquad x\in \mathbb{T}^n
\end{equation}
can be solved by a solution $u^\eps\in C(\mathbb{T}^n,\R)$. As is known, solutions of \eqref{eq:Intro:cell-ergodic} are not unique even up to additive constants \cite{le_dynamical_2017, tran_hamilton-jacobi_2021}, whereas solutions to \eqref{eq:Intro:cell-ergodic-2nd-order} are unique up to adding a constant \cite{lasry_nonlinear_1989}. In the vanishing viscosity process, i.e. $\eps\rightarrow 0^+$, 
the convergence of $u^\varepsilon\to u$ in the full sequence remains unknown, except for special cases \cite{anantharaman_physical_2005, bessi_2003_cmp}. On the other hand, for any fixed $p\in\R^n$, it has been shown in \cite{gomes_stochastic_2002, IMT1} that 
 $\overline{H}^\varepsilon(p) \to \overline{H}(p)$ as $\varepsilon\to 0^+$. Furthermore, the rate of convergence is of order $\mathcal{O}(\varepsilon^{1/2})$ for general nonconvex Hamiltonians \cite[Proposition 5.5]{le_dynamical_2017}. For convex Hamiltonians, the ergodic constant \( \overline{H}(p) \) is of greater interest to experts in dynamical systems, where it is referred to as Mather's \( \alpha \)-function due to its variational meaning in terms of Mather measures \cite{mather_action_1991,mane_generic_1996}. In general, \( p\mapsto \overline{H}(p) \) is only locally Lipschitz in \(\mathbb{R}^n \), although \( p\mapsto \overline{H}^\varepsilon(p) \) is smooth for \( \varepsilon > 0 \) \cite{iturriaga_hector_2005}.

In this paper, we examine the regularity of \( \overline{H}^\varepsilon(p) \) with respect to $\varepsilon$ and $p$ for equation \eqref{eq:Intro:cell-ergodic-2nd-order} with convex Hamiltonians. 
Making use of the scaling structure of Mather measures, we provide a formula for the directional derivative of $\varepsilon \mapsto \overline{H}^\varepsilon(p)$, which can be further applied to obtain the convergent rate $\mathcal{O}(\varepsilon)$ of $\overline{H}^\varepsilon(p) \to \overline{H}(p)$ as $\varepsilon \to 0^+$. Additionally, new formulas for the directional derivatives of \(p \mapsto \overline{H}(p)\) are derived using similar arguments, without assuming the differentiability of \(\overline{H}(p)\). The main contribution comes from utilizing the scaling structure of Mather measures, which offers a new approach to the problem.

\subsection{Assumptions}  We state the main assumptions on the Hamiltonian \(H\) used in this paper. 
\begin{description}[style=multiline, labelwidth=1cm, leftmargin=2cm]
    \item[\namedlabel{itm:assumptions-p1}{$(\mathcal{H}_1)$}] 
        $H(x,\xi)\in C^2(\mathbb{T}^n \times \mathbb{R}^n)$, is superlinear in $\xi$, i.e., $\inf_{x\in\T^n} \frac{H(x,\xi)}{|\xi|}\to +\infty$ as $|\xi|\to \infty$, and $D_{\xi\xi}H(x,\xi)$ is positive definite in $\xi$.

      \item[\namedlabel{itm:assumptions-p2}{$(\mathcal{H}_2)$}] 
    \begin{equation}\label{eq:assumption-coercive-bernstein}
	    	\lim_{|\xi|\to \infty} \inf_{x\in \mathbb{T}^n} \left(\frac{1}{2}H(x,\xi)^2 + D_xH(x,\xi)\cdot \xi \right) = +\infty.
    \end{equation}
\end{description}
Hamiltonians satisfying \ref{itm:assumptions-p1} are known as Tonelli Hamiltonians. Assumption \ref{itm:assumptions-p2} is a growth condition typically required for applying the Bernstein method to obtain gradient bounds for solutions to \eqref{eq:Intro:cell-ergodic-2nd-order}. Under \ref{itm:assumptions-p1}, the Lagrangian \(L \in C^2(\mathbb{T}^n \times \mathbb{R}^n)\) is well-defined, convex, and superlinear in \(v\), and \(D_{vv}L(x,v)\) is positive definite for \(v \in \mathbb{R}^n\), where \(L\) is defined by the Legendre transform:
\begin{equation*}
    L(x,v) := \sup_{\xi \in \mathbb{R}^n} (\xi \cdot v - H(x,\xi)), \qquad (x,v)\in \T^n\times \R^n. 
\end{equation*}

\subsection{Main results}

\begin{thm}\label{thm:regularity-eps} Assume \ref{itm:assumptions-p1} and \ref{itm:assumptions-p2}. Let $c(\varepsilon) := \overline{H}^\varepsilon(0)$, and let $u^\varepsilon$ be any solution to \eqref{eq:Intro:cell-ergodic-2nd-order} with $p=0$. 
\begin{itemize}
	\item[$\mathrm{(i)}$] 
            The function $\varepsilon\mapsto c(\varepsilon)$ is $C^1-$smooth in $\eps>0$ with
    	\begin{equation}\label{eq:formula-c'-eps}
    	\begin{aligned}
    			c'(\varepsilon)  
    			=  -\int_{\mathbb{T}^n \times \mathbb{R}^n} \Delta u^\varepsilon(x)\;d\mu \qquad\text{for all}\;\mu\in \mathcal{M}_0(\varepsilon), 
    	\end{aligned}
    	\end{equation}
	    where $\mathcal{M}_0(\varepsilon)$ is the set of Mather 
            measures associated with \eqref{eq:Intro:cell-ergodic-2nd-order} with $p=0$.
	
	\item[$\mathrm{(ii)}$]

            The function \(\varepsilon \mapsto c(\varepsilon)\) is semiconvex on \((0,1)\) with modulus \(C\varepsilon^{-2}\) for some \(C\) independent of \(\varepsilon\). In particular, \(\varepsilon \mapsto c(\varepsilon)\) is twice differentiable almost everywhere on \((0,1)\), with \(c''(\varepsilon) \geq -C\varepsilon^{-2}\) whenever \(c''(\varepsilon)\) exists. 
            
	\item[$\mathrm{(iii)}$] 
            The function $\varepsilon\mapsto c(\varepsilon)$ is uniformly Lipschitz for $\eps\in [0,1]$, i.e., there exists $C'>0$ independent of $\varepsilon\in (0,1)$, such that
            \begin{equation}\label{eq:corollary:Lipschizt-c-eps-c-0}
            	\big|c(\varepsilon) - c(0)\big|\leq C'\varepsilon.
            \end{equation}
\end{itemize}
\end{thm}

For certain Hamiltonians (e.g., classical mechanical Hamiltonians), the rate \(\mathcal{O}(\varepsilon)\) in \eqref{eq:corollary:Lipschizt-c-eps-c-0} was previously obtained \cite{anantharaman_physical_2005, evans_towards_2004, yu_yifeng_2007_proceeding_ams_a_remark_on}. In these cases, the vanishing viscosity corresponds to the limit from quantum mechanics to classical mechanics as sending the Planck constant to $0$, therefore it is of great interest. In our case we remove these restrictions. We refer to Remark \ref{remark:different-way-c'-eps} for an alternative perspective on \eqref{eq:formula-c'-eps}. 

\begin{thm}\label{thm:direction-derivatives-in-p} Assume \ref{itm:assumptions-p1} and \ref{itm:assumptions-p2}. Then $p\mapsto \overline{H}(p)$ has one-sided directional derivatives in any direction $\xi\in \mathbb{R}^n$, and
	\begin{equation}\label{eq:DpH-bar-0-mu-one-sided}
		D_{\xi+}\overline{H}(p) = \max_{\mu\in \mathcal{M}_p(0)}\int_{\mathbb{T}^n\times \mathbb{R}^n}v\cdot \xi\;d\mu(x,v) \qquad \text{and}\qquad  D_{\xi-}\overline{H}(p) = \min_{\mu\in \mathcal{M}_p(0)}\int_{\mathbb{T}^n\times \mathbb{R}^n}v\cdot \xi\;d\mu(x,v).
	\end{equation}
	Here $\mathcal{M}_p(0)$ is the set of Mather measures associated to \eqref{eq:Intro:cell-ergodic-2nd-order} with $\varepsilon=0$. In particular, if $\mathcal{M}_p(0) = \{\mu\}$ is a singleton, then $p\mapsto \overline{H}(p)$ is differentiable at $p$, with
\begin{equation*}
	D\overline{H}(p) = \int_{\mathbb{T}^n\times \mathbb{R}^n} v\;d\mu(x,v).
\end{equation*}	
\end{thm}

For $\varepsilon>0$, $p\mapsto\overline{H}^\varepsilon(p)$ has been proved to be smooth in $\R^n$ \cite{iturriaga_hector_2005}. Applying Theorem \ref{thm:direction-derivatives-in-p} (the results are actually true for $\varepsilon>0$ in a similar manner) to such a case, we instantly get 
\begin{equation}\label{eq:DpH-bar-eps-mu}
		D\overline{H}^\varepsilon(p) = \int_{\mathbb{T}^n\times\mathbb{R}^n} v\;d\mu(x,v), \qquad\mu\in \mathcal{M}_p(\varepsilon).
\end{equation}
where $\mathcal{M}_p(\varepsilon)$ is the set of Mather measures associated to \eqref{eq:Intro:cell-ergodic-2nd-order}.
The result is obtained through a purely scaling application of Mather measures, as a byproduct of the approach outlined in part (i) of Theorem \ref{thm:regularity-eps}. 

\subsection{Literature} As mentioned earlier, the vanishing viscosity limit from the solution of \eqref{eq:Intro:cell-ergodic-2nd-order} to \eqref{eq:Intro:cell-ergodic} is still a widely open problem, with limited understanding even in the one-dimensional case (convergence is proven under restrictive assumptions in \cite{anantharaman_physical_2005, bessi_2003_cmp}). The study of ergodic constant in view of homogenization for fully nonlinear equation was initiated by Lions, Papanicolaou and Varadhan in 1987.  In their unpublished paper \cite{LPV}, they firstly revealed the existence of 
$\overline{H}(p)$ such that \eqref{eq:Intro:cell-ergodic} solvable. Meanwhile, for Tonelli Hamiltonians, Mather and Ma\~n\'e in \cite{mather_action_1991,mane_generic_1996} propose a dynamical interpretation of the ergodic constant by
\begin{equation}\label{eq:Intro:minimization-problem}
 	-\overline{H}(p) =  \min_{\nu\in \cC} \int_{T\mathbb{T}^n} \Big(L(x,v)-p\cdot v\Big)\;d\nu(x,v),
\end{equation}
where $L(x,v): \T^n\times \mathbb{R}^n\rightarrow\R$ is the Lagrangian, and $\mathcal{C}$ is the set of {\it holonomic measures} (Definition \ref{defn:HolonomicMeasures}). Typically, \(p \mapsto \overline{H}(p)\) is not differentiable, although differentiability is often thought to signal the integrability of the associated Hamiltonian system (see \cite{bolotin_2003} for details). For more relevant works on \(\overline{H}(p)\), see \cite{evans_gomes_2001_effective_I} and the references therein.

For $\varepsilon>0$, the analog of ergodic constant $p\mapsto \overline{H}^\varepsilon(p)$ and the stochastic analog of Mather measures were established in \cite{gomes_stochastic_2002, iturriaga_hector_2005}. Such a case has fine properties, e.g. the stochastic Mather measure is unique for $\eps>0$, and  $p\mapsto \overline{H}^\varepsilon(p)$ is strictly convex and differentiable. The study of the limit $\varepsilon\mapsto \overline{H}^\varepsilon(0)$ appears naturally in the context of vanishing viscosity process. Additionally, an alternative approach for defining stochastic Mather measures under various boundary conditions using duality is developed in \cite{IMT1, IMT2}. In the nonconvex setting, a notion of stochastic Mather measures is introduced in \cite{CagnettiGomesTran2011} using the nonlinear adjoint method, originally developed in \cite{evans_2010_adjoint_and_compensated} and \cite{tran_adjoint_2011}.

In the exploration of problems related to \eqref{eq:Intro:cell-ergodic-2nd-order}, Mather measures can be employed to provide specific criteria for the limits of sequences of Mather measures and viscosity solutions, as demonstrated in \cite{gomes_2008_selection_advcalvar} and \cite{mitake_tran_2017_selection_advance_math}. Our approach, involving scaling measures, shares a similar spirit, and it has been applied in related studies on changing domains, as demonstrated in \cite{bozorgnia_regularity_2024,tu_generalized_2024, tu_vanishing_2022}. From a different perspective, the related question of the regularity issue under perturbation of the Hamiltonian is considered in \cite{gomes_perturbation_2003, aguiar_gomes_regularity_2003}.

\subsection{Organization} Our paper is organized as follows: Section \ref{section:preliminaries} covers preliminary facts about stochastic Mather measures. The proofs of Theorems \ref{thm:regularity-eps} and \ref{thm:direction-derivatives-in-p} are in Sections \ref{section:regularity-eps} and \ref{section:one-sided-regularity-0}, respectively.

\section{Preliminaries}\label{section:preliminaries}

We recall the notion of stochastic Mather measures introduced in \cite{gomes_stochastic_2002}, incorporating certain adjustments for our situation, as detailed below.

\begin{defn} \label{defn:HolonomicMeasures} Let \(\mathcal{P}(\mathbb{T}^n \times \mathbb{R}^n)\) denote the set of probability measures on \(\mathbb{T}^n \times \mathbb{R}^n\). A measure \(\mu \in \mathcal{P}(\mathbb{T}^n \times \mathbb{R}^n)\) is called a \emph{holonomic measure} if 
\begin{equation*}
	\int_{\mathbb{T}^n \times \mathbb{R}^n} |v|\;d\mu(x,v) < \infty,
\end{equation*}
where $|\cdot|$ is the Euclidean norm on $\R^n$ and 
\begin{equation*}
	\int_{\mathbb{T}^n \times \mathbb{R}^n} \big(v\cdot D \varphi(x) - \varepsilon \Delta \varphi(x)\big)\;d\mu(x,v) = 0 \qquad\text{for all}\;\varphi\in C^2(\mathbb{T}^n,\R).
\end{equation*}
We denote by $\mathcal{C}(\varepsilon)$ the set of all holonomic measures associated with $\varepsilon$ in this sense. 
\end{defn}

Suppose $\overline H^\eps(0)$ is the ergodic constant of  \eqref{eq:Intro:cell-ergodic-2nd-order} associated with $p=0$, then

\begin{equation*}
	-\overline{H}^\varepsilon(0) =\inf_{\mu\in\cC(\eps)}	\int_{\mathbb{T}^n \times \mathbb{R}^n} L\;d\mu.
\end{equation*}

Any measure $\mu\in\cC(\eps)$ that attains the infimum is called a \emph{stochastic Mather measure}. The collection of all stochastic Mather measures associated with \eqref{eq:Intro:cell-ergodic-2nd-order} for $p=0$ is denoted by $\mathcal{M}_0(\varepsilon)$ (or associated with $H$). For a general $p\in \mathbb{R}^n$, we use $\mathcal{M}_p(\varepsilon)$ to denote such a set.

\begin{prop}[Sec. 6 of \cite{gomes_stochastic_2002}]
\label{prop:background:unique-measures} 
Assume \ref{itm:assumptions-p1} and \ref{itm:assumptions-p2}. Let $\eps>0$ and $u^\varepsilon$ be a solution of \eqref{eq:Intro:cell-ergodic-2nd-order} with $p=0$. 
\begin{itemize}
	\item[$\mathrm{(i)}$] Any $\mu\in\mathcal{M}_0(\eps)$ is supported in the graph $\{(x,D_pH(x,Du^\varepsilon(x)))\in \T^n\times\R^n~|~x\in\T^n\}$;
	\item[$\mathrm{(ii)}$] 
        By projecting \(\mu(x,v) \in \mathcal{M}_0(\varepsilon)\) onto the \(x\)-coordinates, we obtain a unique invariant density \(\theta^\varepsilon \in W^{1,2}(\mathbb{T}^n,\mathbb{R})\), which satisfies
	\begin{equation*}
            -\mathrm{div} \left(\theta^\varepsilon(x) D_pH\big(x,Du^\varepsilon(x)\big)\right) 
            - \varepsilon \Delta \theta^\varepsilon(x) = 0 \qquad\text{in}\;\mathbb{T}^n. 
	\end{equation*}
\end{itemize}
\end{prop}

Note that the definition of stochastic Mather measures is compatible with the classical Mather measures (\(\varepsilon = 0\)). However, \(\mathcal{M}_0(0)\) is usually not a singleton (see \cite{zhang_zhou_2014_acta_sin}), while \(\mathcal{M}_0(\varepsilon)\) is a singleton for \(\varepsilon > 0\).

\begin{cor}\label{cor:SingletonM0} Assume \ref{itm:assumptions-p1}, \ref{itm:assumptions-p2} and $\varepsilon>0$. The set of Mather measures $\mathcal{M}_0(\varepsilon)$ is a singleton.
\end{cor}
\begin{proof} Take $\mu\in\mathcal{M}_0(\varepsilon)$ and $\phi\in C^0(\T^n\times \R^n)$. By Proposition \ref{prop:background:unique-measures} (i) we have 
\begin{align}\label{eq:SingletonA}
    \int_{\T^n\times\R^n} \phi(x,v)\;d\mu(x,v)  = \int_{\T^n\times \R^n} \phi(x,D_pH(x,Du^\varepsilon(x)))\;d \mu(x,v). 
\end{align}
Let $\theta^\varepsilon$ be unique the projection onto $x$-coordinates of $\mu$ by Proposition \ref{prop:background:unique-measures} (ii). Denote $d\sigma(x)= \theta^\varepsilon(x)dx$ as a measure on $\T^n$, then by slicing measure \cite[Theorem 1.45]{evans_measure_2015}, for $\sigma$-a.e. $x\in \T^n$ there exists a Radon measure $\nu_x$ on $\R^n$ such that $\nu_x(\R^n) = 1$ for $\sigma$-a.e. $x\in \T^n$, and 
\begin{align}
    &\int_{\T^n\times \R^n} \phi(x,D_pH(x,Du^\varepsilon(x)))\;d \mu(x,v) 
    = \int_{\T^n} \left(\int_{\R^n} \phi(x,D_pH(x,Du^\varepsilon(x)))\;d\nu_x(v) \right)\;d\sigma(x) 
    \nonumber \\
    & \qquad\qquad\qquad \qquad \qquad  = \int_{\T^n} \phi(x,D_pH(x,Du^\varepsilon(x))) \;d\sigma(x) = \int_{\T^n} \phi(x,D_pH(x,Du^\varepsilon(x)))\theta^\varepsilon(x)\;dx. 
    \label{eq:SingletonB}
\end{align}
From \eqref{eq:SingletonA} and \eqref{eq:SingletonB}, we see that the action of \(\mu \in \mathcal{M}_0(\varepsilon)\) against any \(\phi \in C^0(\mathbb{T}^n \times \mathbb{R}^n)\) is determined by \(\phi\) and \(\theta^\varepsilon\). Since \(\theta^\varepsilon\) is unique among all measures in \(\mathcal{M}_0(\varepsilon)\) by Proposition \ref{prop:background:unique-measures} (ii), we conclude that \(\mathcal{M}_0(\varepsilon)\) is a singleton. Recall that \(u^\varepsilon\) is unique up to a constant, so \(Du^\varepsilon(x)\) is uniquely determined.
\end{proof}

\begin{lem}[Scaling measures]\label{lem:scale-v-stochastic} Assume \ref{itm:assumptions-p1}, \ref{itm:assumptions-p2} and $\varepsilon>0$. For any measure \(\mu \in \mathcal{P}(\mathbb{T}^n \times \mathbb{R}^n)\) and \(\lambda > 0\), we define the rescaled probability measure \(\mu^\lambda\) by:
\begin{equation*}
	\int_{\mathbb{T}^n \times \mathbb{R}^n} \phi(x,v)\;d\mu^\lambda(x,v) = \int_{\mathbb{T}^n \times \mathbb{R}^n} \phi\big(x,\lambda v\big)\;d\mu(x,v), \qquad \text{for any}\;\phi\in C^0(\mathbb{T}^n \times \mathbb{R}^n).
\end{equation*}
If $\mu \in \mathcal{C}\left(\varepsilon\right) $ then $\mu^\lambda \in \mathcal{C}(\lambda\varepsilon)$. 
\end{lem}
\begin{proof} Let $\varphi\in C^2(\mathbb{T}^n)$, we compute
\begin{equation*}
	\int_{\mathbb{T}^n\times \mathbb{R}^n} \Big(v\cdot D \varphi(x) - (\lambda\varepsilon) \Delta \varphi(x)\Big)\;d\mu^\lambda(x,v) 
	= 
	\lambda\int_{\mathbb{T}^n\times \mathbb{R}^n} \Big( v\cdot D \varphi(x) - \varepsilon\Delta \varphi(x)\Big)\;d\mu(x,v) = 0 
\end{equation*}
since $\mu\in \mathcal{C}(\varepsilon)$.
\end{proof}

\begin{lem}[Bernstein's method, Sec. 1 of \cite{tran_hamilton-jacobi_2021}]\label{lemma:bernstein} Assume \(H(x,\xi) \in C^2(\mathbb{T}^n \times \mathbb{R}^n)\) satisfies \ref{itm:assumptions-p2}. If \(u^\varepsilon \in C^2(\mathbb{T}^n)\) is a solution to \eqref{eq:Intro:cell-ergodic-2nd-order} with \(p = 0\), then \(\|Du^\varepsilon\|_{L^\infty(\mathbb{T}^n)} \leq C\) for some constant \(C\) independent of \(\varepsilon \in (0,1)\).
\end{lem}

Bernstein's method in Lemma \ref{lemma:bernstein} gives an a priori estimate without requiring convexity. By Arzelà-Ascoli Theorem and Lemma \ref{lemma:bernstein}, we obtain the stability of Mather measures and the ergodic constant.

\begin{lem}[Stability]\label{lem:stability-MHeps} Assume \ref{itm:assumptions-p1}, \ref{itm:assumptions-p2}. We have:
\begin{itemize}
	\item[$\mathrm{(i)}$] Let $\varepsilon_k\to \varepsilon$ in $[0,1)$ and $\mu_{\varepsilon_k}\in \mathcal{C}(\varepsilon_k)$. If $\mu_{\varepsilon_k}\rightharpoonup \mu$, then $\mu\in \mathcal{C}(\varepsilon)$.
	\item[$\mathrm{(ii)}$] Let $\varepsilon_k\to \varepsilon$ in $[0,1)$ and $\mu_{\varepsilon_k}\in \mathcal{M}_0(\varepsilon_k)$. If $\mu_{\varepsilon_k}\rightharpoonup \mu$, then $\mu\in \mathcal{M}_0(\varepsilon)$.
	\item[$\mathrm{(iii)}$]  The map $\varepsilon\mapsto\overline{H}^\varepsilon(0)$ is continuous in $[0,1)$.
\end{itemize}
\end{lem}

\section{Proof of Theorem \ref{thm:regularity-eps}: Derivatives with respect to viscosity}\label{section:regularity-eps}
Throughout this section, we consider $p = 0$ in \eqref{eq:Intro:cell-ergodic-2nd-order}. Under \ref{itm:assumptions-p1} and \ref{itm:assumptions-p2}, there is a unique solution \(u^\varepsilon\) to \eqref{eq:Intro:cell-ergodic-2nd-order} with \(u^\varepsilon(0)=0\), and \(\{u^\varepsilon\}_{\varepsilon>0}\) is uniformly Lipschitz by Lemma \ref{lemma:bernstein}. We split the proof into several steps for clarity.
\begin{lem}\label{lem:one-sided} Assume \ref{itm:assumptions-p1} and \ref{itm:assumptions-p2}. The map $\varepsilon\mapsto c(\varepsilon)$ is one-sided differentiable everywhere for $\varepsilon\in (0,1)$, with 
\begin{align}
	   \varepsilon c'_-(\varepsilon) &=  \lim_{\lambda\to 1^-} \left(\frac{c(\lambda \varepsilon) - c(\varepsilon)}{\lambda-1}\right) = \min_{\mu\in \mathcal{M}_0(\varepsilon)} \int_{\mathbb{T}^n \times \mathbb{R}^n} (-v)\cdot D_vL (x,v)\,d\mu, \label{eq:lemma:one-sided-left}\\ 
	   \varepsilon c'_+(\varepsilon) &=  \lim_{\lambda\to 1^+} \left(\frac{c(\lambda \varepsilon) - c(\varepsilon)}{\lambda-1}\right) = \max_{\mu\in \mathcal{M}_0(\varepsilon)} \int_{\mathbb{T}^n \times \mathbb{R}^n} (-v)\cdot D_vL(x,v)\,d\mu.\label{eq:lemma:one-sided-right}
\end{align}
\end{lem}
\begin{proof} 
\noindent Take $\lambda >0$, if $\mu_\lambda \in \mathcal{M}_0(\lambda \varepsilon)$ then the scaling measure $\mu_\lambda^{1/\lambda} \in \mathcal{C}(\varepsilon)$ by Lemma \ref{lem:scale-v-stochastic}, therefore
\begin{equation*}
\begin{aligned}
    &\int_{\mathbb{T}^n \times \mathbb{R}^n} L\big(x,\lambda^{-1}v\big) \,d\mu_\lambda(x,v)=\int_{\mathbb{T}^n \times \mathbb{R}^n} L(x,v)d\mu_\lambda^{1/\lambda}(x,v) \geq -c(\varepsilon),\\
    &\int_{\mathbb{T}^n \times \mathbb{R}^n} L\big(x,v\big) \,d\mu_\lambda(x,v) =  -c(\lambda\varepsilon).
\end{aligned}
\end{equation*}	
We deduce that
\begin{equation*}
	\int_{\mathbb{T}^n \times \mathbb{R}^n} \left(L(x, \lambda^{-1}v) - L(x,v)\right)\,d\mu_\lambda \geq c(\lambda\varepsilon)-c(\varepsilon), \qquad \mu_\lambda\in\mathcal{M}_0(\lambda \varepsilon).
\end{equation*}
If $\lambda < 1$, dividing both sides by $\lambda-1 < 0$ we deduce that 
\begin{align}\label{eq:Lambda1Left}
     \int_{\mathbb{T}^n \times \mathbb{R}^n}\left(\frac{L\left(x, \lambda^{-1}v \right) -L(x,v) }{\lambda-1}\right) \,d\mu_\lambda \leq \frac{c(\lambda\varepsilon) - c(\varepsilon)}{\lambda-1}, \qquad \mu_\lambda\in \mathcal{M}_0(\lambda \varepsilon).
\end{align}
By Lemma \ref{lemma:bernstein}, there exists a constant \(C_0\) such that \(\Vert Du^{\lambda \varepsilon}\Vert_{L^\infty(\T^n)} \leq C_0\) for all \(\lambda \in (0,1)\). Consequently, by Proposition \ref{prop:background:unique-measures} (i), \(\mathrm{supp}(\mu_\lambda) \subset \T^n \times \overline{B}_{R_0}(0)\), where \(R_0 = \max \{D_pH(x, \xi) : x \in \T^n, |\xi| \leq C_0\}\). 
Thus, the sequence of measures \(\{\mu_\lambda\}\) on a compact subset $\T^n\times \overline{B}_{R_0}(0)$ of \(\T^n \times \mathbb{R}^n\) has a weak\(^*\) convergent subsequence $\mu_{\lambda}\rightharpoonup \mu^-$ as $\lambda\to 1^-$ (we use the same notation for the subsequence for simplicity). By Lemma \ref{lem:stability-MHeps} we have $\mu^-\in \mathcal{M}_0(\varepsilon)$, and from \eqref{eq:Lambda1Left} we have
\begin{equation}\label{eq:liminf-lambda-1-stochastic}
    \displaystyle
    -\int_{\mathbb{T}^n \times \mathbb{R}^n} v\cdot D_vL(x,v)\, d\mu^-\leq \liminf_{\lambda\to 1^-} \left(\frac{c(\lambda \varepsilon)-c(\varepsilon)}{\lambda-1}\right) \qquad\text{for some}\;\mu^-\in \mathcal{M}_0(\varepsilon).
\end{equation}
On the other hand, for any $\nu\in \mathcal{M}_0(\varepsilon)$ we have $\nu^{\lambda}\in \mathcal{C}\left(\lambda\varepsilon\right)$ by Lemma \ref{lem:scale-v-stochastic}, thus
\begin{equation}\label{eq:scaling-eps-to-lambda-eps}
\begin{aligned}
	&    \int_{\mathbb{T}^n \times \mathbb{R}^n} L\left(x,\lambda v\right) \, d\nu(x,v) = \int_{\mathbb{T}^n \times \mathbb{R}^n} L\left(x,v\right)d\nu^{\lambda}(x,v)  \geq -c\left(\lambda \varepsilon\right),\\
	&    \int_{\mathbb{T}^n \times \mathbb{R}^n} L\left(x,v\right) \, d\nu(x,v) = -c(\varepsilon).
\end{aligned}
\end{equation}
We deduce that
\begin{equation*}
	\int_{\mathbb{T}^n \times \mathbb{R}^n} \Big(L\left(x,\lambda v\right) -  L(x,v)\Big)\, d\nu \geq c(\varepsilon)- c(\lambda \varepsilon), \qquad\text{for all}\;\nu \in \mathcal{M}_0(\varepsilon).
\end{equation*}
If $\lambda < 1$, dividing both sides by $1-\lambda > 0$ we deduce that 
\begin{equation*}
	-\int_{\mathbb{T}^n \times \mathbb{R}^n} \left(\frac{L\left(x,\lambda v\right) -  L(x,v)}{\lambda-1}\right)\, d\nu \geq \frac{c(\varepsilon) - c(\lambda \varepsilon)}{1-\lambda} \qquad\text{for all}\; \nu \in \mathcal{M}_0(\varepsilon).
\end{equation*}
As $\lambda\to 1^-$ we obtain
\begin{equation}\label{eq:limsup-lambda-1-stochastic}
    \displaystyle
    -\int_{\mathbb{T}^n \times \mathbb{R}^n} v\cdot D_vL(x,v)\,d\nu \geq  \limsup_{\lambda\to 1^-}\left( \frac{c(\varepsilon) -c\left(\lambda\varepsilon\right)}{1-\lambda} \right)  \qquad\text{for all}\;\nu\in \mathcal{M}_0(\varepsilon).
\end{equation} 
From \eqref{eq:liminf-lambda-1-stochastic} and \eqref{eq:limsup-lambda-1-stochastic} we obtain the conclusion \eqref{eq:lemma:one-sided-left}. By a similar argument with $\lambda\to 1^+$ we obtain \eqref{eq:lemma:one-sided-right}.
\end{proof}

\begin{cor}\label{corollary:C1-of-the-map-formula} Assume \ref{itm:assumptions-p1}, \ref{itm:assumptions-p2}, and let $u^\varepsilon$ be any solution to \eqref{eq:Intro:cell-ergodic-2nd-order}. The map $\varepsilon\mapsto c(\varepsilon)$ is in $C^1(0,1)$ with
\begin{equation}\label{eq:corollary:C1-of-the-map-formula}
	c'(\varepsilon) = \varepsilon^{-1}\int_{\mathbb{T}^n \times \mathbb{R}^n} (-v)\cdot D_vL(x,v)\;d\mu(x,v) =  -\int_{\mathbb{T}^n \times \mathbb{R}^n} \Delta u^\varepsilon(x)\;d\mu, \qquad \text{for all}\;\mu\in \mathcal{M}_0(\varepsilon).
\end{equation}
\end{cor}
\begin{proof} For $\mu\in \mathcal{M}_0(\varepsilon)$, if $(x,v) \in\mathrm{supp}(\mu)$ then $(x,v) = (x, D_pH(x,Du^\varepsilon(x)))$, hence by property of Legendre's transform $Du^\varepsilon(x) = D_vL(x,v)$, thus
\begin{equation*}
	\int_{\mathbb{T}^n \times \mathbb{R}^n} v\cdot D_vL(x,v)\, d\mu = 	\int_{\mathbb{T}^n \times \mathbb{R}^n} v\cdot Du^\varepsilon(x) d\mu.
\end{equation*}
Let $\theta^\varepsilon$ be the measure defined in Proposition \ref{prop:background:unique-measures}. Using property of closed measures for $\mu$, we have
\begin{equation*}
	\int_{\mathbb{T}^n \times \mathbb{R}^n} v\cdot Du^\varepsilon(x) d\mu = \int_{\mathbb{T}^n \times \mathbb{R}^n} \varepsilon\Delta u^\varepsilon(x)d\mu = \int_{\mathbb{T}^n} \varepsilon \Delta u^\varepsilon(x)\theta^\varepsilon(x)dx, \qquad\text{for all}\;\mu\in \mathcal{M}_0(\varepsilon).
\end{equation*}
This quantity is independent of $\mu\in \mathcal{M}_0(\varepsilon)$. Consequently, by Lemma \ref{lem:one-sided}, we conclude that $c'_-(\varepsilon) = c'_+(\varepsilon)$, leading to the existence of $c'(\varepsilon)$ and the validity of \eqref{eq:corollary:C1-of-the-map-formula}. Moreover, the continuity of $\varepsilon\mapsto c'(\varepsilon)$ follows from Lemma \ref{lem:stability-MHeps}.
\end{proof}

\begin{rmk}\label{remark:different-way-c'-eps} Our approach to obtain \eqref{eq:formula-c'-eps} is purely on the Lagrangian sides. Another way to view \eqref{eq:formula-c'-eps} is on the PDE side as follows.
As $\varepsilon\mapsto
c(\varepsilon)$ is in $C^1(0,1)$, we can differentiate equation \eqref{eq:Intro:cell-ergodic-2nd-order} with respect to $\varepsilon$ to obtain that
\begin{equation*}
	D_pH(x,Du^\varepsilon)\cdot Du^\varepsilon_\varepsilon - \varepsilon \Delta u^\varepsilon_\varepsilon - \Delta u^\varepsilon = c'(\varepsilon), \qquad x\in \mathbb{T}^n.
\end{equation*}
Here we denote $u^\varepsilon_\varepsilon = \partial_\varepsilon u^\varepsilon$. Using Proposition \ref{prop:background:unique-measures}, integrating this equation against $\mu\in \mathcal{M}_0(\varepsilon)$ we obtain
\begin{equation*}
	\int_{\mathbb{T}^n\times \mathbb{R}^n} \Big(D_pH(x,Du^\varepsilon)\cdot Du^\varepsilon_\varepsilon - \varepsilon \Delta u^\varepsilon_\varepsilon\Big)\;d\mu - \int_{\mathbb{T}^n\times \mathbb{R}^n} \Delta u^\varepsilon\;d\mu = - \int_{\mathbb{T}^n\times \mathbb{R}^n} \Delta u^\varepsilon\;d\mu  = c'(\varepsilon).
\end{equation*}
Here, we exploit the fact that $(x,v)\in \mathrm{supp}(\mu)$ implies $v = D_pH(x,Du^\varepsilon(x))$. Consequently, through the closed measure property, the first term on the left becomes zero, which explains \eqref{eq:formula-c'-eps}.
\end{rmk}

\begin{lem}\label{lem:semiconvex-c-eps} 
Assume \ref{itm:assumptions-p1}, \ref{itm:assumptions-p2}. The map \(\varepsilon \mapsto c(\varepsilon)\) is semiconvex with a linear modulus of order \(\mathcal{O}\left(\varepsilon^{-2}\right)\). Consequently, \(\varepsilon \mapsto c(\varepsilon)\) is twice differentiable for almost every \(\varepsilon \in (0,1)\), and wherever \(c''(\varepsilon)\) exists, we have the following result:
\begin{equation}\label{eq:lem:semiconvex-c-eps}
	c''(\varepsilon)\geq -\left(\int_{\mathbb{T}^n\times\mathbb{R}^n} v\cdot D_{vv}L(x,v)\cdot v\;d\mu(x,v)\right)\varepsilon^{-2}, \qquad \mu\in \mathcal{M}_0(\varepsilon).
\end{equation}
\end{lem}

\begin{proof} If $\lambda > 0$ small then similar to \eqref{eq:scaling-eps-to-lambda-eps}, we have 
\begin{equation*}
\begin{aligned}
	-c((1+\lambda)\varepsilon)  \leq \int_{\mathbb{T}^n \times \mathbb{R}^n} L(x, (1+\lambda)v)\, d\mu  , \qquad
	-c((1-\lambda)\varepsilon)  \leq	\int_{\mathbb{T}^n \times \mathbb{R}^n} L(x, (1-\lambda)v)\, d\mu  
\end{aligned}
\end{equation*}
for all $\mu \in \mathcal{M}_0(\varepsilon)$. Using $\int_{\mathbb{T}^n \times \mathbb{R}^n} L(x, v)\, d\mu = -c(\varepsilon)$ for $\mu \in \mathcal{M}_0(\varepsilon)$ we obtain
\begin{equation}\label{eq:semiconvex-L-equation}
	-\Big(c((1+\lambda)\varepsilon)+c((1-\lambda)\varepsilon)-2c(\varepsilon)\Big) \leq \int_{\mathbb{T}^n \times \mathbb{R}^n} \Big(L(x, (1+\lambda)v) + L(x, (1-\lambda)v) - 2L(x,v)\Big) \,d\mu 
\end{equation}
for any $\mu\in \mathcal{M}_0(\varepsilon)$. Using $\Vert Du^\varepsilon\Vert_{L^\infty(\mathbb{T}^n)} \leq C$ (Lemma \ref{lemma:bernstein}) and $\mathrm{supp}\;\mu \subset \{(x,D_pH(x,Du^\varepsilon(x))): x\in \mathbb{T}^n\}$ from Proposition \ref{prop:background:unique-measures} (i), we see that
\begin{equation*}
	\sup \Big\lbrace \left|v\cdot D_{vv}L(x,sv)\cdot v\right|: (x,v)\in \;\mathrm{supp}\;\mu\;\text{and}\; |s|\leq 1 \Big\rbrace \leq C
\end{equation*}
for some $C$ independent of $\varepsilon$. From \eqref{eq:semiconvex-L-equation} we obtain
\begin{equation}\label{eq:bound-semiconvex-L}
	\int_{\mathbb{T}^n \times \mathbb{R}^n} \Big(L(x, (1+\lambda)v) + L(x, (1-\lambda)v) - 2L(x,v)\Big) \,d\mu \leq C\lambda^2.
\end{equation}
Let \(\lambda = \eta\varepsilon^{-1}\) for \(\eta > 0\). From \eqref{eq:bound-semiconvex-L}, we deduce that \(c(\varepsilon + \eta) + c(\varepsilon - \eta) - 2c(\varepsilon) + \big(C\varepsilon^{-2}\big)\eta^2 \geq 0\) for \(\eta > 0\). Therefore, \(\varepsilon \mapsto c(\varepsilon)\) is semiconvex with modulus \(C\varepsilon^{-2}\). Since a convex function in \(\mathbb{R}\) is differentiable everywhere except for a  set of measure zero, we conclude \eqref{eq:lem:semiconvex-c-eps} from \eqref{eq:semiconvex-L-equation}.
\end{proof}


Next, we show that \(\varepsilon \mapsto c(\varepsilon)\) is uniformly Lipschitz, and that this estimate extends to \(\varepsilon = 0\).

\begin{lem}\label{lem:boundedness-1st-derivative-c-eps} Assume \ref{itm:assumptions-p1}, \ref{itm:assumptions-p2}. Then there exists $C>0$ independent of $\varepsilon$ such that
\begin{equation*}
	|c'(\varepsilon)|\leq C, \qquad\text{for all}\;\varepsilon\in (0,1).
\end{equation*}
As a consequence, $|c(\varepsilon)-c(0)|\leq C\varepsilon$ for all $\varepsilon\in(0,1)$.
\end{lem}

\begin{proof} Differentiating equation \eqref{eq:Intro:cell-ergodic-2nd-order} (for $p=0$) with respect to $x_i$ twice we obtain
\begin{equation*}
\begin{aligned}
	&\Big(D_pH(x, Du^\varepsilon(x))\cdot Du^\varepsilon_{x_ix_i}(x) - \varepsilon\Delta u^\varepsilon_{x_ix_i}(x)\Big) + \left(Du^\varepsilon_{x_i}(x)\right)^T\cdot D_{pp}H(x,Du^\varepsilon(x))\cdot \left(Du^\varepsilon_{x_i}(x)\right) \\
	&\qquad\qquad\qquad\qquad\qquad\qquad + \Big[D_{x_ix_i}H(x,Du^\varepsilon(x)) + 2D_{px_i}H(x,Du^\varepsilon(x))\cdot Du_{x_i}^\varepsilon(x)\Big] = 0, \qquad x\in \mathbb{T}^n.
\end{aligned}
\end{equation*}
From Lemma \ref{lemma:bernstein} (where \ref{itm:assumptions-p2} is necessary) we have $\Vert Du^\varepsilon\Vert_{L^\infty(\mathbb{T}^n)} \leq C$ where $C$ is independent of $\varepsilon>0$. By \ref{itm:assumptions-p1} and the fact that $D_{pp}H(x,p)$ is continuous, there must exists a constant $\gamma>0$ such that 
\begin{equation*}
	\gamma |\xi|^2 \leq \xi^T\cdot D_{pp}H(x,p)\cdot \xi, \qquad \text{for all}\;\xi\neq 0, x\in \mathbb{T}^n, |p|\leq C.
\end{equation*}
Using that we obtain
\begin{equation*}
	\gamma \big|Du^\varepsilon_{x_i}(x)\big|^2  \leq  \left(Du^\varepsilon_{x_i}(x)\right)^T\cdot D_{pp}H(x,Du^\varepsilon(x))\cdot \left(Du^\varepsilon_{x_i}(x)\right) , \qquad x\in \mathbb{T}^n.
\end{equation*}
By Cauchy-Schwartz inequality:
\begin{equation*}
	\big|2D_{px_i}H(x,Du^\varepsilon(x))\cdot Du_{x_i}^\varepsilon(x)\big| \leq \frac{2|D_{px_i}H(x,Du^\varepsilon(x))|^2}{\gamma} + \frac{\gamma}{2} |Du^\varepsilon_{x_i}|^2 \leq \frac{C}{\gamma} + \frac{\gamma}{2} |Du^\varepsilon_{x_i}|^2
\end{equation*}
for $x\in \mathbb{T}^n$. We deduce that for $x\in \mathbb{T}^n$ 
\begin{equation*}
	\big| D_{x_ix_i}H(x,Du^\varepsilon(x)) + 2D_{px_i}H(x,Du^\varepsilon(x))\cdot Du_{x_i}^\varepsilon(x) \big| \leq C + \frac{C}{\gamma} + \frac{\gamma}{2} |Du^\varepsilon_{x_i}|^2
\end{equation*}
for a constant $C$ depends only on $H$. Putting everything together we deduce that 
\begin{equation*}
	\Big(D_pH(x, Du^\varepsilon(x))\cdot Du^\varepsilon_{x_ix_i}(x) - \varepsilon\Delta u^\varepsilon_{x_ix_i}(x)\Big) + \frac{\gamma}{2} \big| Du^\varepsilon_{x_i}(x)\big| ^2 \leq C, \qquad x\in \mathbb{T}^n.
\end{equation*}
Taking integration against a viscosity Mather measure $\mu\in \mathcal{M}(\varepsilon)$, noting that $(x,v)\in \mathrm{supp}(\mu)$ if and only if $v = D_pH(x,Du^\varepsilon(x))$ due to Proposition \ref{prop:background:unique-measures}, we obtain 
\begin{equation*}
	\int_{\mathbb{T}^n \times \mathbb{R}^n}\Big(D_pH(x, Du^\varepsilon(x))\cdot Du^\varepsilon_{x_ix_i}(x) - \varepsilon\Delta u^\varepsilon_{x_ix_i}(x)\Big) \;d\mu = 0
\end{equation*}
due to $\mu$ is a closed measure in $\mathcal{C}(\varepsilon)$, therefore
\begin{equation*}
	\frac{\gamma}{2}\int_{\mathbb{T}^n \times \mathbb{R}^n} \big| Du^\varepsilon_{x_i}(x)\big| ^2 \;d\mu\leq C \qquad\Longrightarrow\qquad \int_{\mathbb{T}^n \times \mathbb{R}^n} \big| D^2 u^\varepsilon(x)\big| ^2 \;d\mu \leq \frac{Cn}{\gamma}.
\end{equation*}
As $\mu$ is a probability measures on $\mathbb{T}^n \times \mathbb{R}^n$, by Corollary \ref{corollary:C1-of-the-map-formula} we have
\begin{equation*}
	|c'(\varepsilon)| \leq \int_{\mathbb{T}^n \times \mathbb{R}^n} |\Delta u^\varepsilon(x)|\;d\mu \leq \left(\int_{\mathbb{T}^n \times \mathbb{R}^n} |\Delta u^\varepsilon(x)|^2\;d\mu \right)^{1/2} \leq \left(\int_{\mathbb{T}^n \times \mathbb{R}^n} |D^2 u^\varepsilon(x)|^2\;d\mu \right)^{1/2}  \leq \left(\frac{Cn}{\gamma}\right)^{1/2}.
\end{equation*}
Therefore the conclusion $|c(\varepsilon) - c(0)|\leq C\varepsilon$ follows form the fundamental theorem of calculus.
\end{proof}

\begin{rmk} For mechanical Hamiltonians $H(x,\xi) = |\xi|^p - V(x)$ with $p>1$ and $V\in C(\mathbb{T}^n)$, its Lagrangian satisfies $v\cdot D_vL(x,v) \geq 0$ for all $(x,v)\in \mathbb{T}^n\times \mathbb{R}^n$, therefore $c'(\varepsilon) \leq 0$ by Corollary \ref{corollary:C1-of-the-map-formula}, thus Lemma \ref{lem:boundedness-1st-derivative-c-eps} gives 
\begin{equation*}
	c(\varepsilon)  \leq c(0) \leq c(\varepsilon) + C\varepsilon
\end{equation*}
for $\varepsilon>0$. This has been already observed in \cite[Theorem 5.2 (i)]{evans_towards_2004} for $H(x,\xi) = \frac{1}{2}|\xi|^2 - V(x)$.
\end{rmk}

\begin{proof}[Proof of Theorem \ref{thm:regularity-eps}] 
The aforementioned Lemmas \ref{lem:one-sided}, \ref{corollary:C1-of-the-map-formula}, \ref{lem:semiconvex-c-eps}, and \ref{lem:boundedness-1st-derivative-c-eps} cover Theorem \ref{thm:regularity-eps}.
\end{proof}

\section{Proof of Theorem \ref{thm:direction-derivatives-in-p}: Directional derivatives of the effective Hamiltonian}\label{section:one-sided-regularity-0}

Similar to Proposition \ref{prop:background:unique-measures} (i), we note that for \(\varepsilon = 0\), it is well-known that classical Mather measures are supported on a compact subset of \(\mathbb{T}^n \times \mathbb{R}^n\). 

\begin{lem}\label{lem:MatherEpsZeroSupport} Assume \ref{itm:assumptions-p1}, \ref{itm:assumptions-p2}, and $\varepsilon=0$. Then any $\mu\in \mathcal{M}_p(0)$ satisfies
\begin{equation}\label{eq:SupportMatherA}
    \mathrm{supp}(\mu)\subset \left\{(x,v)\in \T^n\times \R^n: H(x,p+D_vL(x,v)) = \overline{H}(p)\right\}. 
\end{equation}
Consequently, if $|p|\leq R$ then there exists $C_R>0$ such that $\mathrm{supp}(\mu) \subset \T^n\times \overline{B}_{C_R}(0)$, a compact subset of $\T^n\times \R^n$. 
\end{lem}

\begin{proof} For a proof of \eqref{eq:SupportMatherA}, we refer to \cite{fathi_weak_2008}, \cite[Lemma 7.13]{tran_hamilton-jacobi_2021} or \cite[Theorems 3 and 37]{biryuk_introduction_2010}. For $p\in \R^n$ with $|p|\leq R$, we denote $\mathcal{S}_p= \{(x,v)\in \T^n\times \R^n: H(x,p+D_vL(x,v)) = \overline{H}(p)\}$. Since $p\mapsto \overline{H}(p)$ is continuous \cite{tran_hamilton-jacobi_2021}, we can find $C_1(R)>0$ such that $\overline{H}(q)\leq C_1(R)$ for all $|q|\leq 2R$. By \ref{itm:assumptions-p1}, there exists $C_2(R)>0$ such that if $(x,v)\in \mathcal{S}_p$ then
\begin{equation*}
    |p+D_vL(x,v)| \leq C_2(R)\qquad\Longrightarrow\qquad |D_vL(x,v)|\leq |p|+C_2(R) \leq R+C_2(R).
\end{equation*}
Since \(D_{vv}L(x,v)\) is positive definite in \(v\), as per \ref{itm:assumptions-p1}, there exists a constant \(C_R > 0\) such that 
\begin{equation*}
     |D_vL(x,v)| \leq R + C_2(R)\qquad\Longrightarrow\qquad 
     |v| \leq C_R.
\end{equation*}
In other words, we obtain that $\mathrm{supp}(\mu)\subset \mathcal{S}_p\subset \T^n \times \overline{B}_{C_R}(0)$ if $\mu\in \mathcal{M}_p(0)$ and $|p|\leq R$. 
\end{proof}

\begin{proof}[Proof of Theorem \ref{thm:direction-derivatives-in-p}] For $\varepsilon = 0$, $\mathcal{C}(0)$ is the set of holonomic measures, i.e., probability measures such that 
\begin{equation*}
	\int_{\mathbb{T}^n \times \mathbb{R}^n} v\cdot D\phi(x) \;d\mu(x,v) = 0 \qquad\text{for all}\;\phi\in C^1(\mathbb{T}^n).
\end{equation*}
For $p\in \mathbb{R}^n$, the Legendre transform of $(x,\xi)\mapsto H(x,p+\xi)$ is the \emph{shifted} Lagrangian $L(x,v) -p\cdot v$ for $(x,v)\in \mathbb{T}^n \times \mathbb{R}^n$. We have 
\begin{equation}\label{eq:thm:one-sided-derivatives-defn-Mane-Mather-eps=0}
	 -\overline{H}(p) = \inf_{\mu\in \mathcal{C}(0)} \int_{\mathbb{T}^n\times \mathbb{R}^n} \big( L(x,v)-p\cdot v\big)\;d\mu.
\end{equation}
Let $\mathcal{M}_p(0)$ be the set of Mather measures associated to \eqref{eq:Intro:cell-ergodic-2nd-order} with $\varepsilon=0$, i.e., those measures in $\mathcal{C}(0)$ that minimize \eqref{eq:thm:one-sided-derivatives-defn-Mane-Mather-eps=0}. Let $\xi\in \mathbb{R}^n$ be a direction, we compute the one-sided derivatives $D_{\xi\pm}\overline{H}(p)$. We have
\begin{equation*}
\begin{cases}
     \begin{aligned}
         &\int_{\mathbb{T}^n \times \mathbb{R}^n} \Big(L(x,v) - p\cdot v\Big)\;d\mu(x,v) = -\overline{H}(p),\\
         &\int_{\mathbb{T}^n \times \mathbb{R}^n} \Big(L(x,v) - (p+t\xi)\cdot v\Big)\;d\mu(x,v) \geq -\overline{H}(p+t\xi),
     \end{aligned}
\end{cases} \qquad \text{for all}\;\mu \in \mathcal{M}_p(0). 
\end{equation*}
We deduce that 
\begin{equation}\label{eq:mu_Mp}
     \overline{H}(p+t\xi)-\overline{H}(p)  \geq t\int_{\mathbb{T}^n \times \mathbb{R}^n} \xi \cdot v\;d\mu(x,v)\qquad \text{for all}\;\mu \in \mathcal{M}_p(0).
\end{equation}
On the other hand, choose a sequence $\mu_t \in \mathcal{M}_{p+t\xi}(0)$ for $t>0$, then 
\begin{equation*}
\begin{cases}
     \begin{aligned}
         &\int_{\mathbb{T}^n \times \mathbb{R}^n} \Big(L(x,v) - p\cdot v\Big)\;d\mu_t(x,v) \geq -\overline{H}(p),\\
         &\int_{\mathbb{T}^n \times \mathbb{R}^n} \Big(L(x,v) - (p+t\xi)\cdot v\Big)\;d\mu_t(x,v) = -\overline{H}(p+t\xi).
     \end{aligned}
\end{cases}
 \end{equation*}
Therefore 
\begin{equation}\label{eq:mu_M-p-t}
     \overline{H}(p+t\xi) -\overline{H}(p) \leq t\int_{\mathbb{T}^n \times \mathbb{R}^n} \xi \cdot v\;d\mu_t(x,v), \qquad \mu_t \in \mathcal{M}_{p+t\xi}(0).
\end{equation}
From \eqref{eq:mu_Mp} as $t\to 0^+$, we obtain 
\begin{equation}\label{eq:directional_liminf}
	\liminf_{t\to 0^+} \frac{\overline{H}(p+t\xi) - \overline{H}(p)}{t} \geq \max_{\mu\in \mathcal{M}_p(0)} \int_{\mathbb{T}^n \times \mathbb{R}^n} \xi \cdot v\;d\mu(x,v).
\end{equation}

For \(t \in (0,1)\) small enough, we have $|p+t\xi| \leq R$ for some $R>0$, and thus \(\mu_t\) is supported on a common compact subset of \(\T^n \times \R^n\) by Lemma \ref{lem:MatherEpsZeroSupport}. Hence, \(\{\mu_t\}\) has a weak\(^*\) convergent subsequence, \(\mu_t \rightharpoonup \mu^+\) as \(t \to 0^+\). By Lemma \ref{lem:stability-MHeps}, \(\mu^+ \in \mathcal{M}_p(0)\). Using $\mu_{t}\rightharpoonup \mu^+$ as $t\to 0^+$ in \eqref{eq:mu_M-p-t}, we obtain 
\begin{equation}\label{eq:directional_limsup}
    \limsup_{t\to 0^+} \frac{\overline{H}(p+t\xi) - \overline{H}(p)}{t} \leq  \int_{\mathbb{T}^n \times \mathbb{R}^n} \xi \cdot v\;d\mu^+(x,v), \qquad \mu^+ \in \mathcal{M}_p(0).
\end{equation}
From \eqref{eq:directional_liminf} and \eqref{eq:directional_limsup}, we obtain
\begin{equation*}
	D_{\xi+}\overline{H}(p)= \lim_{t\to 0^+} \frac{\overline{H}(p+t\xi) -\overline{H}(p)}{t} = \max_{\mu\in \mathcal{M}_p(0)} \int_{\mathbb{T}^n \times \mathbb{R}^n} \xi \cdot v\;d\mu(x,v).
\end{equation*}
By a similar argument, as $t\to 0^-$, we obtain 
\begin{equation*}
\begin{aligned}
     D_{\xi-}\overline{H}(p)= &\lim_{t\to 0^-} \frac{\overline{H}(p+t\xi) -\overline{H}(p)}{t} = \min_{\mu\in \mathcal{M}_p(0)} \int_{\mathbb{T}^n \times \mathbb{R}^n} \xi \cdot v\;d\mu(x,v).
 \end{aligned}
\end{equation*}
Finally, if \(\mathcal{M}_p(0) = \{\mu\}\) is a singleton, we deduce that \(p \mapsto \overline{H}(p)\) has a directional derivative in every direction \(\xi \in \mathbb{R}^n\).  Since $p\mapsto \overline{H}(p)$ is convex \cite[Theorem 4.9]{tran_hamilton-jacobi_2021}, its subgradient $\partial \overline{H}(p) \neq \emptyset$ for any $p\in \R^n$. We show that $\partial \overline{H}(p)$ is a singleton. Indeed, if $q_1,q_2 \in \partial \overline{H}(p)\subset\mathbb{R}^n$, then for any direction $\xi \in \mathbb{R}^n$, we have 
\begin{equation*}
	D_\xi \overline{H}(p) = q_1\cdot \xi = q_2\cdot\xi = \int_{\T^n\times \R^n} \xi\cdot v\;d\mu(x,v). 
\end{equation*}
This holds for all \(\xi \in \mathbb{R}^n\); therefore, \(q_1 = q_2\), and thus \(p \mapsto \overline{H}(p)\) is differentiable at \(p\).
\end{proof}

\section*{Acknowledgments} The authors extend their gratitude to Professors Hung Tran, Hiroyoshi Mitake, Yifeng Yu, and Diogo Gomes for valuable discussions, insightful comments, and for sharing relevant references. They also thank the anonymous referee for numerous helpful suggestions that enhanced the clarity of the paper.

\bibliography{refs.bib}{}
\bibliographystyle{acm}

\end{document}